\documentclass[12pt]{article}
\RequirePackage{amssymb,amsmath,amscd,amsthm}
\usepackage{enumitem}
\usepackage{url}
\usepackage{authblk}
\usepackage{fancyhdr}
\newtheorem{thm}{Theorem}[section]

\newtheorem{cor}[thm]{Corollary}
\newtheorem{lemma}[thm]{Lemma}
\newtheorem{rem}[thm]{Remark}
\DeclareMathOperator{\cod}{cod}
\DeclareMathOperator{\Irr}{Irr}

\DeclareMathOperator{\cd}{cd}

\newtheorem*{Main theorem}{Main theorem}
\begin{document}
	
	\title{On $p$-groups with one  nonlinear codegree }
	\pagestyle{fancy}
	\fancyhf{}                     
	\fancyhead[C]{Oone  nonlinear codegree} 
	\fancyfoot[C]{\thepage}  
	\renewcommand{\headrulewidth}{0pt}  
	
	\author
	[a]{Jiwei Zheng}
	\author
	[a]{Heng Lv\thanks{Corresponding author:
			\texttt
			{lvh529@163.com}}}
	{
		\affil
		[a]{School of Mathematics and Statistics,
			Southwest University,
			Chongqing 400715, P.R. China}
		\maketitle	
\begin{abstract}
	For an irreducible character $\chi$ of a finite group $G$, its codegree is
	defined by $\cod(\chi)=|G:\ker\chi|/\chi(1)$. In this paper, we characterize fnite $p$-groups whose nonlinear irreducible characters all have the same codegree.
\end{abstract}


\section{Introduction}

Let $G$ be a finite group, and let $\Irr(G)$ denote the set of its
irreducible complex characters. Following Qian et al.
\cite{qian2007co}, the codegree of $\chi\in\Irr(G)$ is defined by
$$
\cod(\chi)=\frac{|G:\ker\chi|}{\chi(1)}.
$$
We write $\cod(G)=\{\cod(\chi)|\chi\in\Irr(G)\}$. Restrictions on character codegrees therefore provide a way
to relate character-theoretic information to the structure of finite
groups. In particular, Alizadeh et al.~\cite{alizadeh2019groups}
characterized finite groups with at most three character codegrees. Then Liu and Yang \cite{liu2021nonsolvable4,liu2021nonsolvable5} determined the nonsolvable finite groups with exactly four or five character codegrees.

For finite $p$-groups, an important aspect of this problem is the relation
between character codegrees and nilpotency class. Du and Lewis
\cite{DuLewis2016} obtained bounds on the nilpotency class in terms of the
largest character codegree. Finite $p$-groups with exactly three character
codegrees have nilpotency class at most $2$, and Li and Qu
\cite{Li02102024} characterized these groups. Croome and Lewis
\cite{croome2020} investigated $p$-groups with exactly four character
codegrees and established bounds on their nilpotency class under
additional hypotheses.

Instead of restricting the size of $\cod(G)$, one can impose conditions
only on the codegrees of nonlinear irreducible characters. For a normal
subgroup $N$ of $G$, we identify $\Irr(G/N)$ with the irreducible characters
of $G$ whose kernels contain $N$, and put
$$
\Irr(G| N)=\Irr(G)\setminus\Irr(G/N),
\qquad
\cod(G| N)=\{\cod(\chi)|\chi\in\Irr(G| N)\}.
$$
Thus $\cod(G| G')$ consists precisely of the codegrees of nonlinear irreducible
characters of $G$. Qian and Zeng \cite{qian2023finite} characterized the
non-nilpotent finite groups satisfying $|\cod(G| G')|=1$. $c(G)$ denotes the nilpotence class of $G$. In this paper, prove the following
result.
\begin{Main theorem}\label{main}
	Let $G$ be a nonabelian $p$-group. Then $|\cod(G|G')|=\{p^a\}$ if and only if
	$c(G)=2$, $G'$ is elementary abelian, and the following two conditions hold
	for every maximal subgroup $N$ of $G'$, where $Z(G/N)=R_N/N$,
	$|G:R_N|=p^{2d_N}$ and $\exp(R_N/N)=p^{e_N}$:
	\begin{enumerate}[label=\textup{(\roman*)}]
		\item $G'/N\leq \mho_{e-1}(Z(G/N))$;
		\item $d_N+e_N=a$ for an integer $a$ which does not depend on $N$.
	\end{enumerate}
	
\end{Main theorem}

\section{Proof of main Theorem}
In the proof of the main theorem, we  need only the following  two basic result on characters.

\begin{lemma}[\cite{isaacs1994character}, Corollary 2.30]\label{230}
Let $\chi\in\Irr(G)$. Then $\chi(1)^2\leq|G:Z(\chi)|$. Equality occurs iff $\chi$ vanishes on $G-Z(\chi)$.
\end{lemma}

\begin{lemma}[\cite{isaacs1994character}, Theorem 2.31]\label{231}
	Suppose that $\chi\in\Irr(G)$ and that $G/Z(\chi)$ is abelian. Then $|G:Z(\chi)|=\chi(1)^2$.
\end{lemma}

For a
$p$-group $G$, we write $G_i$ for the $i$th term of the lower central series of $G$, defined by $G_1=G$ and $G_{i+1}=[G_i,G]$. And  we write $\mho_t(G)=\{x^{p^t}| x\in G\}$.
\begin{lemma}\label{cod2}
Let $G$ be a $p$-group such that $|G_3|=p$ and $Z(G)$ is cyclic. 
Then $|\cod(G| G')|\geq 2$.
\end{lemma}	
	\begin{proof}
		Let $\chi\in \Irr(G)$ be faithful and $\psi\in \Irr(G/G_3)$ be nonlinear. Write $\exp(Z(G/G_3))=p^z$ and $K=\ker\psi$.
		
		\smallskip
		\noindent\emph{Step 1.} $\cod(\chi)^2\geq p^z|G|$.
		Let $M/G_3=Z(G/G_3)$. For $x\in M$ and $g\in G$, we have
		$[x,g]\in G_3\leq Z(G)$, and hence $[x^p,g]=[x,g]^p=1$. Thus $\mho_1(M/G_3)\subseteq Z(G)/G_3$ which is cyclic. Then $|Z(G)/G_3|\geq \exp(\mho_1(M/G_3))=p^{z-1}$. Hence, $|Z(G)|\geq p^z$. By Lemma \ref{230}, we have
		\[\cod(\chi)^2=\frac{|G|^2}{\chi(1)^2}\geq |G|\,|Z(G)|\geq p^z|G| .
		\]

		\noindent\emph{Step 2.} $\cod(\psi)^2\leq p^{z-1}|G|$.
		Let $\bar G=G/G_3$. Since $Z(\psi)/K=Z(G/K)$ is
		cyclic, there exists $w\in Z(\psi)$ such that $Z(\psi)=\langle w\rangle K$.
		For every $g\in G$, we have $[\bar w,\bar g]\in\bar K\cap\bar G'$, and
		$[\bar w,\bar g]$ is central in $\bar G$. Let
		$p^t=\exp(\bar K\cap\bar G')$. Then
		$[\bar w^{\,p^t},\bar g]=[\bar w,\bar g]^{p^t}=1$ for all $g\in G$, that is,
		$\bar w^{\,p^t}\in Z(\bar G)$. Hence $\bar w^{\,p^{t+z}}=1$, so that
		$w^{p^{t+z}}\in G_3\leq K$. Then
		\[
		|Z(\psi):K|=o(wK)\leq p^{t+z}\leq p^z|\bar K|=p^{z-1}|K| .
		\]
		By Lemma \ref{231}, $\psi(1)^2=|G:Z(\psi)|$. Thus we have
		\[\cod(\psi)^2=\frac{|G:K|^2}{\psi(1)^2}=|G:K||Z(\psi):K|\leq p^{z-1}|G|.
		\]
		
		Therefore, $\cod(\chi)\geq p\cod(\psi)$. Hence $|\cod(G| G')|\geq 2$.

	\end{proof}

\begin{cor}\label{cG2}
	Let $G$ be a $p$-group such that $|\cod(G| G')|=1$. Then $c(G)=2$.
\end{cor}
\begin{proof}
	Let G be a counterexample of minimal order. Let $N$ be a maximal sugroups of $G_3$ and $\chi\in \Irr(G/N)$ be nonlinear. Then $$|\cod(G/\ker(\chi)N| (G/\ker(\chi)N)')|=1.$$ Thus $\ker(\chi)N=1$. It follows that $|G_3|=p$ and $Z(G)$ is cyclic. By Lemma \ref{cod2}, $|\cod(G| G')|\geq 2$. This completes the proof.
\end{proof}

\begin{lemma}\label{ele}
	Let $G$ be a $p$-group such that $|\cod(G| G')|=1$. Then $G'$ is elementary abelian.
\end{lemma}
\begin{proof}
By the Corollary \ref{cG2}, $c(G)=2$. Assume that $G$ is a counterexample of minimal order. It is clearly that $\exp(G')=p^2$. Since 
\[ G\lesssim \times_{\chi\in \Irr(G)} G/\ker(\chi),
\] there is a $\chi\in \Irr(G)$ such that $|G'\ker(\chi)/\ker(\chi)|=p^2$. From the minimality of $G$, we have $\ker(\chi)=1$. Thus $|G'|=p^2$ and $Z(G)$ is cyclic.

Let $N$ be the maximal subgroup of $G'$ and $\psi\in\Irr(G/N)$ be nonlinear. By Lemma \ref{231}, $\psi(1)^2=|G:Z(\psi)|$. Let $K=\ker(\psi)$. Then 
\[ \cod(\psi)^2=\frac{|G:K|^2}{\psi(1)^2}=|G:K||Z(\psi):K|.
\]
Since $Z(\psi)/K$ is cyclic, there exists $w\in Z(\psi)$ such that
$Z(\psi)=\langle w\rangle K$. For every $g\in G$, we have
$[w,g]\in K\cap G'=N$, and so $[w^p,g]=[w,g]^p=1$. Thus $w^p\in Z(G)$. As
$N\leq Z(G)\cap K$, the cyclic group $Z(G)K/K\cong Z(G)/(Z(G)\cap K)$ has
order at most $|Z(G)|/p$. Therefore
\[
|Z(\psi):K|=o(wK)\leq p\,o(w^pK)\leq p\cdot\frac{|Z(G)|}{p}=|Z(G)| .
\]
By Lemma \ref{230}, we have
\[ \cod(\chi)^2=\frac{|G|^2}{\chi(1)^2}\geq |G|\,|Z(G)|.
\]
Consequncely,
\[
\cod(\psi)^2\leq |G:K|\,|Z(G)|=\frac{|G|\,|Z(G)|}{|K|}
\leq \frac{|G|\,|Z(G)|}{p}<\cod(\chi)^2,
\]
which contradicts $|\cod(G| G')|=1$.
\end{proof}

\begin{lemma}\label{G'=p}
	Let $G$ be a $p$-group with $|G'|=p$, and let $\exp(Z(G))=p^{e}$. Then
	\[
	\cod(G|G')=\bigl\{\,|G:Z(G)|^{1/2}\,o(\lambda)\bigm|
	\lambda\in\Irr(Z(G)),\ G'\not\leq\ker\lambda\,\bigr\}.
	\]
	Consequently, $|\cod(G|G')|=1$ if and only if $G'\leq\mho_{e-1}(Z(G))$. In this case $\cod(G|G')=\{|G:Z(G)|^{1/2}p^{e}\}$.
\end{lemma}
\begin{proof}
Since $|G'|=p$, we have $G'\leq Z(G)$. Let $\chi\in \Irr(G)$ be nonlinear, and let $\lambda\in\Irr(Z(G))$ be such that $\chi_{Z(G)}=\chi(1)\lambda$. By Problrm 2.13 in \cite{isaacs1994character}, we have $\chi(1)^2=|G:Z(G)|$. It follows that $Z(\chi)=Z(G)$, and hence $\ker(\chi)=\ker(\lambda)$. Then we have

\[
\cod(\chi)=\frac{|G:Z(G)|\,|Z(G):\ker\lambda|}{\chi(1)}
=|G:Z(G)|^{1/2}\,o(\lambda).
\]
Conversely, if $\lambda\in\Irr(Z(G))$ and $G'\not\leq\ker\lambda$, then every
irreducible constituent $\chi$ of $\lambda^G$ satisfies
$\chi_{Z(G)}=\chi(1)\lambda$ and is nonlinear. This proves the displayed
formula.

Write $Y=\Irr(Z(G))$ and $Y_0=\{\lambda\in Y\mid G'\leq\ker\lambda\}$. If $G'\leq \mho_1(Z(G))$, there exists a
$x\in Z(G)$ such that $x^{p^{e-1}}$ generates $G'$. Then
$\lambda(x)^{p^{e-1}}=\lambda(x^{p^{e-1}})\neq 1$ for every
$\lambda\in Y\setminus Y_0$, and hence $o(\lambda)=p^e$. Conversely, assume
that all characters in $Y\setminus Y_0$ have the same order. Choose
$\lambda_0\in Y$ of order $p^e$ and $\mu\in Y\setminus Y_0$. If $o(\mu)<p^e$,
then one of $\lambda_0$ and $\lambda_0\mu$ lies in $Y\setminus Y_0$ and has
order $p^e$, a contradiction. Thus every $\lambda\in Y$ with
$\lambda^{p^{e-1}}=1$ lies in $Y_0$. Since
$\lambda^{p^{e-1}}(y)=\lambda(y^{p^{e-1}})$ for $y\in Z(H)$, these are exactly
the characters of $Z(H)$ whose kernels contain $\mho_{e-1}(Z(H))$. Hence $H'$
is contained in the intersection of their kernels, which is
$\mho_{e-1}(Z(H))$.
\end{proof}
\begin{proof}[Proof of Main theorem] 
	Assume that $\cod(G)=\{p^a\}$. By Corollary \ref{cG2} and Lemma \ref{ele}, we have that $c(G)=2$ and $G'$ is elementary. Let $N$ be a maximal subgroup of $G'$. Then $\cod(G/N)=\cod(G)=\{p^a\}$. By Lemma \ref{G'=p}, we obtain  (i) and (ii). Conversely, let $\chi\in \Irr(G)$ be nonlinear and $N=\ker\chi$. Since $c(G)=2$ and $G'$ is elementary, $|G'N/N|=p$. By Lemma \ref{G'=p}, $\cod(\chi)=p^{d_N+e_N}=p^a$. 

\end{proof}

\begin{cor}\label{cor:degrees}
Let $G$ be a $p$-group with $\cod(G|G')=\{p^a\}$, let $\exp(Z(G))=p^z$, and put $s=a-z$. Then $\cd(G)\subseteq\{1,p^{s-1},p^{s},p^{s+1}\}$ and $\cd(G)\neq\{1,p^{s+1}\}$.
		
\end{cor}
\begin{proof}
By the main theorem, $c(G)=2$ and $G'$ is elementary abelian.
Let $\chi\in\Irr(G)$ be nonlinear, and put $N=\ker\chi\cap G'$.
Since $G'/N\cong G'\ker\chi/\ker\chi\leq Z(G/\ker\chi)$,
and $Z(G/\ker\chi)$ is cyclic, the nontrivial elementary abelian
group $G'/N$ has order $p$. By Lemma~\ref{G'=p}, we have
\[
p^a=\cod(\chi)=\chi(1)\exp(Z(G/N)).
\]
Since $N\leq Z(G)$ has exponent at most $p$ and
$Z(G)/N\leq Z(G/N)$, we obtain $p^{z-1}\leq\exp(Z(G)/N)\leq\exp(Z(G/N)).$
On the other hand, $[x^p,g]=[x,g]^p=1$ for all $x,g\in G$,
so $x^p\in Z(G)$ for every $x\in G$. Consequently,
$\exp(Z(G/N))\leq p^{z+1}$. It follows that $ p^{a-z-1}\leq\chi(1)\leq p^{a-z+1}$. Hence, $\cd(G)\subseteq\{1,p^{s-1},p^s,p^{s+1}\}$.

Suppose that $\cd(G)=\{1,p^{s+1}\}$. Then $e_N=z-1$ for every maximal
subgroup $N$ of $G'$, so that $\exp(Z(G)/N)\leq p^{z-1}$, that is,
$\mho_{z-1}(Z(G))\leq N$. As $G'$ is elementary abelian, the intersection of
its maximal subgroups is trivial, and hence $\mho_{z-1}(Z(G))=1$. This
contradicts $\exp(Z(G))=p^z$.
\end{proof}

\begin{rem}\label{rem:patterns}
	Let $G$, $a$, $z$ and $s$ be as in Corollary~\ref{cor:degrees}. By
	Corollary~\ref{cor:degrees}, the set $\cd(G)\setminus\{1\}$ is one of the six
	nonempty subsets of $\{p^{s-1},p^s,p^{s+1}\}$ different from $\{p^{s+1}\}$.
	For $p=2$, four of these patterns are realized by the following groups
	of the SmallGroups library (identified by GAP):
	\begin{itemize}[label={}]
		\item $\cd(G)\setminus\{1\}=\{p^s\}$: $D_8$;
		\item $\cd(G)\setminus\{1\}=\{p^{s-1}\}$: $\mathrm{SmallGroup}(32,33)$;
		\item $\cd(G)\setminus\{1\}=\{p^{s-1},p^s\}$: $\mathrm{SmallGroup}(64,219)$;
		\item $\cd(G)\setminus\{1\}=\{p^{s},p^{s+1}\}$: $\mathrm{SmallGroup}(128,2213)$.
	\end{itemize}
	
	The bound $|\cd(G)|\leq 4$ is attained. For example, let $G$ have generators
	$x_1$, $y_1$, $x_2$, $y_2$, $x_3$, $y_3$, $c$ and $n$, subject to the relations
	$$
	\begin{gathered}
		c^4=n^2=1,\qquad x_1^2=c,\qquad
		y_1^2=x_2^2=y_2^2=x_3^2=y_3^2=1,\\
		[x_1,y_1]=[x_2,y_2]=n,\qquad [x_3,y_3]=nc^2,
	\end{gathered}
	$$
	with all other commutators between generators trivial. Then $|G|=2^9$,
	$Z(G)=\langle c,n\rangle\cong C_4\times C_2$, and
	$G'=\langle c^2,n\rangle\cong C_2\times C_2$. Put $A=Z(G)$.
	For the three maximal subgroups $N$ of $G'$, the notation of
	Main theorem gives
	$$
	\begin{array}{c|c|c|c|c}
		N & R_N & d_N & e_N & x_N\\ \hline
		\langle c^2\rangle & A & 3 & 1 & n\\
		\langle n\rangle & \langle A,x_1,y_1,x_2,y_2\rangle & 1 & 3 & x_1\\
		\langle nc^2\rangle & \langle A,x_3,y_3\rangle & 2 & 2 & c
	\end{array}
	$$
	In each row, $x_N^{2^{e_N-1}}\in G'\setminus N$ and $d_N+e_N=4$.
	Thus Main theorem yields
	$$
	\cod(G\mid G')=\{16\},\qquad \cd(G)=\{1,2,4,8\}.
	$$
	Here $a=4$, $z=2$ and $s=2$, so all three possible nonlinear degrees
	$2^{s-1},2^s,2^{s+1}$ occur.
\end{rem}

\end{document}